\documentclass[a4paper,11pt]{article}
\usepackage[a4paper,bindingoffset=0.2in,%
left=1in,right=1in,top=1in,bottom=1in,%
           footskip=.25in]{geometry}
\usepackage{graphicx}

\usepackage{amssymb}
\usepackage{amsthm}
\usepackage{amsmath}
\usepackage{color}
\usepackage[utf8]{inputenc}
\usepackage{tikz}
\usepackage[all]{xy}
\usepackage[bookmarksnumbered,colorlinks]{hyperref}
\usepackage[normalem]{ulem}
\usepackage{float} % % Figuras
\usepackage{authblk}

\usepackage[euler-digits]{eulervm}
\usepackage{enumerate}
\def\br#1\er{\textcolor{red}{#1}} %
\newcommand{\be}{\begin{equation}}
\newcommand{\ee}{\end{equation}}

\newcommand{\ben}{\begin{enumerate}}
\newcommand{\een}{\end{enumerate}}
\newcommand{\bit}{\begin{itemize}}
\newcommand{\eit}{\end{itemize}}
\newcommand{\edoc}{\end{document}}

\newcommand{\cambios}[1]{{#1}}

\newcommand{\qcd}{\begin{flushright} $\Box$ \end{flushright}}

\begin{document}
\newtheorem{thm}{Theorem}[section]
\newtheorem{prop}[thm]{Proposition}
\newtheorem{lemma}[thm]{Lemma}
\newtheorem{cor}[thm]{Corollary}
\newtheorem{conv}[thm]{Convention}
\theoremstyle{definition}
\newtheorem{defi}[thm]{Definition}
\newtheorem{notation}[thm]{Notation}
\newtheorem{exe}[thm]{Example}
\newtheorem{conj}[thm]{Conjecture}
\newtheorem{prob}[thm]{Problem}
\newtheorem{rem}[thm]{Remark}
%\theoremstylx

\usetikzlibrary{matrix}
\date{}

\title{Some remarks on conformal symmetries and Bartnik's splitting conjecture}

\author[1]{I.P. Costa e Silva}
%\email{pontual.ivan@ufsc.br}
%\address{Department of Mathematics,\\
%University of Miami, Coral Gables, FL 33124, USA.}
\author[2]{J.L. Flores}
%\email{floresj@uma.es}
%\address{Departamento de \'Algebra, Geometr\'{i}a y Topolog\'{i}a,\\ Facultad de Ciencias, Universidad de M\'alaga \\ Campus Teatinos, 29071 M\'alaga, Spain}
\author[3]{J. Herrera}
%\address{Department of Mathematics,
%Universidade Federal de Santa Catarina, 88.040-900 Florian\'{o}polis-SC, Brazil.}
%\address{Universidade Federal de Santa Catarina, Brazil}
%\address{Universidad de M\'alaga, Spain}
%\date{March 16th, 2016}
%\newcommand{\qed}{\nobreak \ifvmode \relax \else
%	\ifdim\lastskip<1.5em \hskip-\lastskip
%	\hskip1.5em plus0em minus0.5em \fi \nobreak
%	\vrule height0.75em width0.5em depth0.25em\fi}
%\newcommand{\qcd}{\begin{flushright} $\Box$ \end{flushright}}
%\usepackage[cp850]{inputenc}
%\usepackage{amsfonts,a4a4}
\affil[1]{\small{{\it Department of Mathematics\\
Universidade Federal de Santa Catarina, 88.040-900 Florian\'{o}polis-SC, Brazil.}}}
\affil[2]{\small{{\it Departamento de \'Algebra, Geometr\'{i}a y Topolog\'{i}a\\ Facultad de Ciencias, Universidad de M\'alaga\\ Campus Teatinos, 29071 M\'alaga, Spain.}}}
\affil[3]{\small{{\it Departamento de Matemáticas, Edificio Albert Einstein, Universidad de Córdoba\\ Campus de Rabanales, 14071 Córdoba, Spain.}}}

\maketitle

\begin{abstract}
Inspired by the results in a recent paper by G. Galloway and C. Vega \cite{GV2}, we investigate a number of geometric consequences of the existence of a timelike conformal Killing vector field on a globally hyperbolic spacetime with compact Cauchy hypersurfaces, especially in connection with the so-called Bartnik's splitting conjecture. In particular we give a complementary result to the main theorem in \cite{GV2}.
\end{abstract}

%  \flushbottom

%\tableofcontents

\section{Introduction: motivations \& statement of main results}\label{S0}
   In 1988, R. Bartnik \cite{B} posed the following conjecture:
   \begin{conj}\label{conj: bartnik} {\em Let $(M^{n+1},g)$ be a globally hyperbolic spacetime with compact Cauchy hypersurfaces, satisfying the {\em timelike convergence condition} (TCC) $Ric(v,v) \geq 0$ for every $v \in TM$ timelike. Then, either $(M,g)$ is timelike geodesically incomplete or else it is globally isometric to a product spacetime $(\mathbb{R}\times S, -dt^2 \oplus h)$, where $(S,h)$ is a compact Riemannian manifold.}
   \end{conj}

 Bartnik calls a globally hyperbolic spacetime (of any dimension) possessing a compact Cauchy hypersurface and satisfying the TCC a {\em cosmological spacetime}, because the ``spatially closed'' Robertson-Walker models in relativistic cosmology (with suitable matter content and cosmological constant) are obviously important examples of such spacetimes. The Bartnik's conjecture then becomes the statement that {\em a cosmological spacetime is either timelike geodesically incomplete or else splits isometrically as a product of a Lorentz line and a compact Riemannian manifold}. (For simplicity, we shall simply say then that the pertinent spacetime {\em splits}; in this paper this phrase will always refer to the specific kind of isometric splitting appearing in the Bartnik conjecture.)

The importance of Bartnik's conjecture both for geometry and physics lies in that it establishes a rigidity statement for the celebrated 1970 singularity theorem of Hawking and Penrose \cite{HP,yau}. The latter theorem implies, in particular, that cosmological spacetimes which in addition satisfy the {\em generic condition} have an incomplete timelike or null geodesic. Of course, it is well-known (see, e.g., Chs. 2 and 14 of \cite{BE}) that the main technical effect of the generic condition is to induce pairs of conjugate points along complete nonspacelike geodesics; hence it is natural to ask whether nonspacelike geodesic incompleteness still holds when one drops the generic condition. In this context, Bartnik's conjecture (if true) implies that timelike geodesic completeness indeed holds only in exceptional, non-generic cases in the class of cosmological spacetimes.

Bartnik's conjecture has been investigated by a number of researchers, especially G. Galloway and collaborators \cite{B,EG,G1,G2,GV1,GV2}, who have proven alternative versions under a variety of additional assumptions. Two of the latter will especially concern us here. The first one is an early result by Galloway \cite{G1}: if a cosmological spacetime $(M,g)$ is timelike geodesically complete and satisfies the {\em no-observer-horizon (NOH) condition}
\begin{equation}
\label{NOH}
M= I^{\pm}(\gamma) \mbox{ for any inextedible timelike curve $\gamma$},
\end{equation}
then $(M,g)$ splits as in the Bartnik conjecture.

The NOH condition is relevant for the Bartnik's conjecture lies in that it implies that $(M,g)$ has {\em no null geodesic lines}. To see this, let such a null geodesic line $\gamma$ be given. In this case, we would have $I^-(\gamma) \cap Im \gamma = \emptyset$ from the achronality of $\gamma$. But $I^-(\gamma)$ is an {\em indecomposable terminal past set} (TIP) \cite{GKP}, and as such there would exist an inextendible {\em timelike} curve $\beta$ \cite{GKP} such that $I^-(\beta) = I^-(\gamma) \neq M$, violating the NOH condition.

Now, any globally hyperbolic spacetime $(M,g)$ with compact Cauchy hypersurfaces admits either a null or a timelike geodesic line by standard arguments (see, e.g., Ch. 8 in \cite{BE}). Since the first situation is excluded if the NOH condition holds, we would then have a timelike line in $(M,g)$. But a timelike geodesically complete cosmological spacetime with a timelike line necessarily splits by the Lorentzian Splitting Theorem \cite{E,EG,BE}.

The second germane partial result is much more recent \cite{GV2}, and assumes the presence of a {\em timelike conformal Killing vector field} in a {\em vacuum} cosmological spacetime:

\begin{thm}
\label{GVThm2}
[Galloway \& Vega, '17] Let $n\geq 2$ and suppose $(M^{n+1},g)$ is a Ricci-flat, timelike geodesically complete cosmological spacetime possessing a timelike conformal Killing vector field $X \in \mathfrak{X}(M)$. Then $X$ is in fact a Killing vector field and $(M,g)$ splits.
\end{thm}

Our goal in this paper is to clarify a little further the effects of the presence of a timelike conformal Killing vector field on the global geometry of spacetimes with compact Cauchy hypersurfaces. In view of Theorem \ref{GVThm2} and our own results below, this investigation is directly relevant to Bartnik's conjecture, and this is our main motivation; but a distinctive feature of our work here is that we also pursue results {\em independently of any assumptions on the Ricci tensor}.

%Our pursuit of a broader study stem partly from physical considerations. Indeed, the {\em direct} physical relevance of the original context of Bartnik's conjecture for cosmology might be called into question, since solid empirical evidence \cite{perlmutter,riess} now seems to indicate an accelerated expansion in our universe, usually attributed to some form of ``dark energy'' whose density at the present epoch largely dominates that of overall matter, including dark matter (see, e.g., \cite{DE,DM} for reviews and the references therein). This accelerated expansion, in turn, suggests that the TCC does not hold in cosmological scales\footnote{One might also argue that present empirical data also favor an ``open'', (i.e. non-compact) spatially flat universe \cite{flat}, and hence that spatial compactness would also be a dated assumption. However, to the best of our knowledge, cosmological data of the large-scale {\em topology} of the universe are nowhere as conclusive as those on the accelerated expansion. In any case, spatial compactness remains an important technical assumption in this paper.}. Therefore, a more general analysis becomes cogent.

On the one hand, all Friedmann-Lema\^itre-Robertson-Walker geometries which are so highly useful in cosmology admit a timelike conformal Killing vector field (usually not complete). This fact alone justifies pursuing a deeper understanding of the effects of the presence of that form of conformal symmetry in the geometric setting of Bartnik's conjecture, not necessarily satisfying the TCC. On the other hand, we need not make any special assumptions about spatial isotropy/homogeneity in this paper.

Concretely, our first main result is to show that the main role of a {\em complete} timelike conformal Killing vector field is precisely to ensure that the NOH condition (and hence the absence of null geodesic lines) is satisfied:

\begin{thm}
\label{mainthm1}
Suppose $(M,g)$ is a globally hyperbolic spacetime with compact Cauchy hypersurfaces, and admits a complete timelike conformal Killing vector field $X \in \mathfrak{X}(M)$. Then the no-observer-horizon condition (\ref{NOH}) is satisfied in $(M,g)$.
\end{thm}

We emphasize that this result in particular does not require the TCC, and indeed no condition on the Ricci tensor at all. Its proof is actually elementary after one recalls some well-known facts about stationary spacetimes; we include it here as a theorem both because it is important in the development of the ideas in this paper, but also because the result itself seems not to have been stated before explicitly in the literature. Its relevance in the Bartnik's context, if the TCC {\em is} assumed, is given by the following

\begin{cor}
\label{maincor1}
Let $(M,g)$ be a timelike geodesically complete cosmological spacetime. Then the following are equivalent.
\begin{itemize}
\item[i)] $(M,g)$ splits.
\item[ii)] $(M,g)$ admits a timelike Killing vector field.
\item[iii)] $(M,g)$ admits a complete timelike Killing vector field.
\item[iv)] $(M,g)$ admits a complete timelike conformal Killing vector field.
\end{itemize}
\end{cor}

Corollary \ref{maincor1} shows that the existence {\em and completeness} of some (conformal) Killing vector field are inescapable if the Bartnik's conjecture is to hold. Indeed, in the proof of Theorem \ref{GVThm2} as presented in \cite{GV2}, the crucial role of Ricci-flatness lies in showing that $X$ is Killing. (More precisely, one needs that ${\cal L}_XRic \equiv 0$, where ${\cal L}_X$ denotes the Lie derivative with respect to the conformal Killing field $X$.) But once one does ensure that $X$ is Killing, its completeness then follows from a result by Garfinkle and Harris \cite{GH}, which is used both in \cite{GV2} and to prove that $(ii)$ implies $(iii)$ in Corollary \ref{maincor1} (see section \ref{S2} below).

The particular argument given in \cite{GV2}, although very elegant, seems hard to generalize (in a natural way) for more general Ricci tensors. This has in part motivated our investigations here.

Now, without the TCC it is perfectly possible for a spacetime to be geodesically complete, globally hyperbolic with compact Cauchy hypersurfaces and yet to possess an {\em incomplete} conformal Killing vector field. A simple and very important example is given by de Sitter spacetime, which is globally a warped product $(\mathbb{R}\times \mathbb{S}^n, -dt^2 \oplus a ^2  \omega_n)$, where $(\mathbb{S}^n,\omega_n)$ is the unit radius, round $n$-sphere, and the warping function is $a: t \in \mathbb{R} \mapsto \cosh t \in \mathbb{R}$. In this case, $X = a\partial_t$ is an incomplete conformal Killing vector field. Note, however, that $Ric(X,X) <0$, i.e. TCC is violated everywhere. Thus, de Sitter spacetime is {\em not} cosmological in Bartnik's technical sense, and hence it is still consistent both with Bartnik's conjecture and with Corollary \ref{maincor1}.

However, it is well-known (see, e.g. \cite{HE}) that de Sitter spacetime (as well as other FLRW geometries) can be {\em conformally embedded} as an open set in a larger spacetime (namely, Einstein's static universe $(\mathbb{R}\times \mathbb{S}^n, -dt^2 \oplus \omega_n)$) in which $X$ can be extended as a {\em complete} unit timelike {\em Killing} vector field. Our second main result is that this latter construction can be generalized as follows:

\begin{thm} \label{mainthm2} Suppose $(M^{n+1},g)$ is globally hyperbolic, and let $X$ be a timelike conformal Killing vector field in $(M,g)$. Then $(M,g)$ admits an open conformal embedding $\varphi: (M,g) \hookrightarrow (\hat{M}^{n+1}, \hat{g})$ with a {\em complete, unit timelike Killing vector field} $\hat{X} \in \mathfrak{X}(\hat{M})$ extending the pushforwarded $\varphi_{\ast} X$ in $\hat{M}$. Moreover, if $S \subset M$ is a smooth, spacelike (hence acausal) compact Cauchy hypersurface for $(M,g)$, then $(\hat{M}, \hat{g})$ can be chosen so that $\varphi(S)$ is also a smooth, spacelike (hence acausal) compact Cauchy hypersurface for $(\hat{M}, \hat{g})$.
\end{thm}

Again, no assumptions on the Ricci tensor are made, and compactness of the Cauchy hypersurfaces is only required for the second part.

%Theorem \ref{mainthm2} yields the following corollary.

%\begin{cor}
%\label{maincor2}
%Suppose $(M^{n+1},g)$ is a timelike geodesically complete cosmological spacetime possessing a conformal Killing vector field $X \in \mathfrak{X}(M)$. With the notation in Theorem \ref{mainthm2}, if $(M,\hat{g}|_{M})$ also satisfies the TCC, then $(M,g)$ splits.
%\end{cor}

Our last main result extends Theorem \ref{GVThm2} to a broader context insofar as we assume the TCC but not Ricci-flatness. However,
in order to prove the splitting result we have been unable to do without imposing an additional geometric condition to control the asymptotic behavior of the scalar curvature and the conformal Killing vector field. This control need only be enforced along one timelike geodesic.

\begin{thm}
\label{mainthm4}
Assume that the following holds for the globally hyperbolic spacetime $(M^{n+1},g)$ with $n \geq 2$:
\begin{itemize}
\item[i)] the TCC is satisfied and $(M,g)$ has compact Cauchy hypersurfaces (so that $(M,g)$ is cosmological);
\item[ii)] $(M,g)$ is timelike geodesically complete;
\cambios{\item[iii)] The timelike conformal Killing vector field $X$ satisfies at least one of the following conditions:
\begin{itemize}
\item[iii.1)] $X$ is complete, or
\item[iii.2)] for some future-complete timelike geodesic $\gamma:[0,+\infty) \rightarrow M$ there exists a number $A>0$ such that
\begin{equation}
\label{bound}
(\sigma\circ \gamma)^2 - \frac{(\beta \circ \gamma)^2}{n(n-1)}R\circ \gamma \leq A \mbox{ along $\gamma$},
\end{equation}
where $R$ is the scalar curvature of $(M,g)$, and $\beta := \sqrt{- g(X,X)}$, and $\sigma \in C^{\infty}(M)$ is as in Eq. (\ref{conformalfield}) (see below), or
\item[iii.3)] $\sigma$ is bounded.
\end{itemize}
}
\end{itemize}
Then $(M,g)$ splits as in the Bartnik conjecture.
\end{thm}

\begin{rem}
\label{imprmk}
 The latter result is compatible with Theorem \ref{GVThm2} in the following concrete sense. If $Ric=0$ then we immediately have $R=0$, and the results in Ref. \cite{GV2} imply that $\sigma \equiv 0$. \cambios{Therefore, the bounds in items $(iii.2)$ and $(iii.3)$ of Theorem \ref{mainthm4} are automatically satisfied.}
\end{rem}

The rest of the paper is organized as follows. In section \ref{S1} we give some basic technical preliminaries, mainly to establish the terminology and notation. In section \ref{S2} the proofs of Theorem \ref{mainthm1} and Corollary \ref{maincor1} are given, while in \ref{S3} we prove Theorem \ref{mainthm2}. In section \ref{S4}, however, not only Theorem \ref{mainthm4} itself is proven, but we discuss in detail the geometric structure of the conformal embedding described in Theorem \ref{mainthm2} when the conformal Killing vector field $X$ is incomplete. A number of ancillary results are proven therein which are of independent interest.

\section{Preliminaries}\label{S1}

Throughout this paper, we shall fix once and for all a {\em spacetime}, i.e., a connected, Hausdorff, second-countable time-oriented $C^{\infty}$ Lorentzian manifold $(M^{n+1},g)$ with $dim \, M = n+1 \geq 2$. We shall assume that the reader is familiar with standard facts in Lorentz geometry and causal theory as given in the basic references \cite{BE,HE,oneill}.

In this section, we establish some preliminary technical facts which will be useful in the proof of our main theorems. We also recall a few standard definitions in order to settle the precise terminology and notation we shall use throughout the paper.

Let $Z \in \mathfrak{X}(M)$ be a timelike vector field (not necessarily conformal Killing). Its maximal integral curves define a $1$-dimensional foliation of $M$. We shall denote the set of leaves of this foliation by $Q = Q_Z$, and the standard projection which takes each $p\in M$ to the leaf through $p$ by $\pi = \pi_Z: M \rightarrow Q$. Then we have the following

\begin{prop}
\label{first}
Suppose $(M,g)$ is globally hyperbolic and let $Z \in \mathfrak{X}(M)$ be a timelike vector field. Then, the following facts hold.
\begin{itemize}
\item[i)] There exists a unique topology and differentiable structure on $Q$ which makes it a smooth $n$-manifold for which $\pi:M \rightarrow Q$ is a smooth submersion. Moreover, $M$ is (non-canonically) diffeomorphic to $\mathbb{R}\times Q$. (In particular, $Q$ is homeomorphic to any given Cauchy hypersurface $S \subset M$.)
\item[ii)] If, in addition, $Z$ is Killing, then there exists a unique Riemannian metric $h$ on $Q$ such that the canonical projection $\pi:M \rightarrow Q$ is a semi-Riemannian submersion, and the metric $g$ can then be written as
\begin{equation}
\label{stationarymetric}
g = g(Z,Z)\theta_Z\otimes \theta_Z + \pi^{\ast} h,
\end{equation}
where
\[
\theta_Z := g(\frac{Z}{g(Z,Z)},\, .\, ).
\]
\end{itemize}
\end{prop}
\begin{proof}$(i)$\\
Pick any positive $f \in C^{\infty}(M)$ for which $Y:= f\cdot Z$ is complete. The maximal integral curves of $Y$ and $Z$ clearly define the same foliation, so in particular $Q_Y = Q_Z = Q$. Let $\phi: \mathbb{R}\times M \rightarrow M$ be the flow of $Y$. Since $Y$ is complete, $\phi$ is a smooth action of the abelian group $(\mathbb{R},+)$ on $M$. By Corollary 3.14 of \cite{splittingus}, the action $\phi$ is free and proper. With respect of this action, $\pi$ then defines on $M$ the structure of a (necessarily trivial) principal $\mathbb{R}$-bundle over $Q$. See, e.g., in \cite{Lee}, p. 218, Thm. 9.16 of \cite{Lee}, Sections 12 and 13 of \cite{Vandenban}, and Theorem 5.7, p. 58 of \cite{KN} for more details. \\
$(ii)$\\
Recall (see, e.g., Ch. 77, p. 212, Def. 44 of \cite{oneill}) that $\pi:(M,g) \rightarrow (Q,h)$ being a semi-Riemannian submersion means that $\forall x \in Q$, the fiber $\pi^{-1}(x)$ is a semi-Riemannian submanifold of $M$, which here is a trivial requirement, and $d\pi$ preserves scalar products in vectors normal to fibers (`horizontal spaces'). The latter condition can be used to {\em define} $h$: it is well-defined because since $Z$ is Killing, horizontal spaces are preserved along the integral curves thereof. The decomposition (\ref{stationarymetric}) then follows from a direct computation.
\end{proof}

%IRROTATIONAL
%
%Recall that a timelike vector field $Z \in \mathfrak{X}(M)$ (again, not necessarily conformal Killing) is said to be {\em irrotational} if the orthogonally complementary distribution $Z^{\perp} \subset TM$ is integrable. The following standard facts can be immediately deduced from this definition.
% \begin{itemize}
% \item[1)] Irrotationality is a {\em conformally invariant} notion: if $Z$ is irrotational in $(M,g)$, then it is also irrotational in $(M,\Omega^2g)$, where $\Omega \in C^{\infty}(M)$ is any smooth positive real-valued function.
% \item[2)] $Z$ is irrotational if and only if $\forall\, V,W \in \mathfrak{X}(M)$,
% \[
% V,W \perp Z \Longrightarrow [V,W] \perp Z.
% \]
% \item[3)] $Z$ is irrotational if and only if the {\em curl} $2$-form $curl \, Z \in \Omega^2(M)$, given by
% \[
% curl \, Z(V,W) := g(\nabla _V Z,W) - g(\nabla_W Z,V), \forall\, V,W \in \mathfrak{X}(M),
% \]
% vanishes for $V,W$ in $Z^{\perp}$.
% \item[4)] $Z$ is irrotational if and only if the unit field $Z/(-g(Z,Z))^{1/2}$ is irrotational.
% \end{itemize}

We shall say that $(M,g)$ is {\em conformastationary} if there exists some $X \in \mathfrak{X}(M)$ which is everywhere a {\em timelike conformal Killing vector field}, i.e.,
 \begin{equation}
 \label{conformalfield}
 {\cal L}_Xg = 2 \sigma g
 \end{equation}
 for some function $\sigma \in C^{\infty}(M)$, where here and hereafter ${\cal L}_X$ denotes the Lie derivative with respect to $X$. Of course, $X$ is Killing if and only if $\sigma \equiv 0$, in which case we say that $(M,g)$ is {\em stationary} (with respect to $X$). From (\ref{conformalfield}) we easily deduce that
 \begin{equation}
 \label{conformalfunction}
 \sigma:= \frac{X(g(X,X))}{2 g(X,X)}.
 \end{equation}

 {\em In this paper, we {\em do not} take a given conformal Killing vector field $X$ on $(M,g)$ to be complete, unless otherwise explicitly stated}. On the other hand we do require it to be timelike everywhere\footnote{The reader should be aware that such a completeness assumption is often included as part of the definition of `conformastationary (resp. stationary) spacetime' in the literature. Moreover, sometimes the timelike character is imposed only in some asymptotic sense, especially in the study of stationary black holes.}.

%$(M,g)$ is said to be {\em static} if it admits an {\em irrotational} timelike Killing vector field $X \in \mathfrak{X}(M)$.

\begin{exe}(Standard (conforma)stationary spacetimes)\\
\label{ex1}
Let $(M_0,g_0)$ be any smooth Riemannian $n$-manifold. On $M_0$, pick a smooth, real-valued, strictly positive function $\beta_0$, and a smooth 1-form $\omega_0 \in \Omega^1(M_0)$. Fix also a strictly positive smooth function $\Lambda_0 \in C^{\infty}(\mathbb{R}\times M_0)$. Then, the {\em standard conformastationary spacetime} associated with the data $(M_0,g_0,\beta_0,\omega_0,\Lambda_0)$ is $(M, g)$, where $M:= \mathbb{R} \times M_0$, and
\begin{equation}
\label{standardstationary}
g = \Lambda_0^2(-\beta ^2 d\pi_1 \otimes d\pi_1 + \omega \otimes d\pi_1 + d\pi_1 \otimes \omega + \pi_2^{\ast}g_0),
\end{equation}
where $\beta := \beta_0 \circ \pi_2$, $\omega := \pi_2^{\ast} \omega_{0}$, and $\pi_1$ [resp. $\pi_2$] is the projection of $M$ onto the $\mathbb{R}$ [resp. $M_0$] factor. The time-orientation of $(M,g)$ is chosen such that $\partial_t$, the lift to $M$ of the standard vector field $d/dt$ on $\mathbb{R}$, is future-directed. The vector field $\partial_t$ is then a timelike conformal Killing vector field and $(M,g)$ is indeed conformastationary. If $\Lambda_0\equiv 1$, then $(M,g)$ is said to be {\em standard stationary} (for the respective data), and if in addition $\omega_{0} \equiv 0$, then $(M,g)$ thus defined is said to be {\em standard static}.

The following facts about the standard conformastationary metric (\ref{standardstationary}) will be useful to us here:
\begin{itemize}
\item[1)] The timelike conformal Killing vector field $\partial_t$ is complete.
\item[2)] $\pi_1$ is a smooth {\em temporal function}, i.e., it has timelike gradient \cite{JS}, and hence the hypersurfaces $\{t\} \times M_0$ are acausal and spacelike for each $t \in \mathbb{R}$. In particular, the spacetime is stably causal \cite{JS}.
\end{itemize}
\end{exe}

Next, consider the following well-known result, obtained in a broader context in Ref. \cite{JS}, and with a much simplified proof in the particular case of globally hyperbolic spacetimes in \cite{CFS}:

\begin{prop}
\label{firstt}
 Let $(M,g)$ be a globally hyperbolic spacetime admitting a complete timelike conformal Killing vector field $X \in \mathfrak{X}(M)$. Then, $(M,g)$ is isometric to a standard conformastationary spacetime. Indeed, with the notation as in (\ref{standardstationary}), $M_0$ can be chosen to be a compact Cauchy hypersurface.
\end{prop}

Recall that $P \subset M$ [resp. $F \subset M$] is a {\em past set} [resp. {\em future set}] if $P = I^-(P)$ [resp. $F = I^+(F)$]. Note that past and future sets are always open. A non-empty set $A \subset M$ is said to be an {\em achronal boundary} if $A = \partial P$ for some past set $P$. In this case, it is easy to check that $F:= M \setminus \overline{P}$ is a future set and $A \equiv \partial F$, so we may alternatively define an achronal boundary as the - non-empty - boundary of a {\em future} set. Indeed, as discussed in detail in Ref. \cite{GV1}, for an achronal boundary $A$ there exists a {\em unique} disjoint decomposition
\[
M = P \dot{\cup} F \dot{\cup} A
\]
where $P$ is a past set, $F$ is a future set, and $A = \partial P = \partial F$. In particular, any achronal boundary {\em separates} $M$, i.e., $M \setminus A$ is  not connected.

As its name suggests, an achronal boundary $A$ is always an achronal edgeless set\footnote{$A$ is then a closed $C^0$ (indeed Lipschitz) hypersurface in $M$, see e.g., Corollary 26, Ch. 14 in \cite{oneill}.}, but the converse is not true in general. The following result, however, exploits an exception which will be of key importance later on and has independent interest.
\begin{prop}
\label{jsmain1}
Suppose $(M,g)$ is chronological (i.e., admits no closed timelike curves) and let $X \in \mathfrak{X}(M)$ be a {\em complete} timelike conformal Killing vector field with flow $\phi$. Then, any (non-empty) achronal edgeless set $A \subset M$ is an achronal boundary, and each orbit of $\phi$ intersects $A$ exactly once. Moreover,
\[
\zeta= \phi |_{\mathbb{R} \times A}: (t,p) \in \mathbb{R} \times A \mapsto \phi_t(p) \in  M
\]
is a homeomorphism. In particular, $A$ is connected and separates $M$. If $(M,g)$ is globally hyperbolic, then $A$ is homeomorphic to any given Cauchy hypersurface $S \subset M$; and if in addition $A$ is compact then it {\em is} a Cauchy hypersurface.
\end{prop}
\begin{proof} The map $\zeta$ is continuous (smooth if $A$ is smooth), and one-to-one since $A$ is achronal. Since $A$ is a $C^0$ hypersurface, Invariance of Domain implies that $\zeta$ is then a homeomorphism onto a open subset ${\cal O} \subset M$, and each orbit of $\phi$ intersects $A$ at most once. Showing that each such orbit does indeed intersect $A$ and that $\zeta$ is a homeomorphism onto $M$ boil down therefore to showing that ${\cal O} \equiv M$. Since ${\cal O}$ is open and $M$ is connected, all we need to show is that ${\cal O}$ is closed. To this end, consider a sequence $(t_k, x_k)$ in $\mathbb{R} \times A$ and $p \in M$ such that $\phi_{t_k}( x_k) \rightarrow p$.

Assume first that $(t_k)$ is unbounded. We may assume, up to passing to a subsequence that $t_k \rightarrow +\infty$, the argument if $t_k \rightarrow -\infty$ being analogous. Let $\gamma$ be any maximal integral curve of the timelike conformal Killing vector field $X$. Since $X$ is assumed to be complete, we can apply the Corollary 3.2 of \cite{HL} to conclude that
\begin{equation}
\label{omniscience}
M = I^{\pm}(\gamma).
\end{equation}
In particular, taking $\gamma$ to be the orbit of $p$ by $\phi$ we have that $\phi_s(p) \in I^-(x_1)$ for some $s \in \mathbb{R}$. ($x_1$ being the first term in the sequence $(x_k)$!). But then
\[
\phi_{t_k+s}(x_k) = \phi_s(\phi_{t_k}(x_k)) \rightarrow \phi_s(p),
\]
so for large enough $k$ we have $t_k +s >0$ and $\phi_{t_k+s}(x_k) \in I^-(x_1)$; hence,
\[
x_k \ll \phi_{t_k+s}(x_k) \ll x_1,
\]
which contradicts the achronality of $A$. Therefore, $(t_k)$ must be bounded. But in that case, up to passing to a subsequence we may assume that it converges, say, $t_k \rightarrow t_0$. Let $x_0:= \phi_{-t_0}(p)$. Then
\[
x_k = \phi(-t_k,\phi_{t_k}(x_k)) \rightarrow x_0,
\]
and since $A$ is closed we conclude that $x_0 \in A$, and $\zeta(t_0,x_0) = p$, which shows that ${\cal O}$ is closed, as desired.

To see that $A$ is an achronal boundary, let $P := I^-(A)$. $P$ is clearly a past set, and the fact that $A$ is achronal implies that $A \subset \partial P$. Given any $p \in \partial P$, the previous results show that $\phi_t(p) \in A$ for some $t \in \mathbb{R}$. But $\partial P$ is achronal (it is an achronal boundary!), so we must have $t \equiv 0$, which means that $p \in A$. We conclude that $A \equiv \partial P$, so that $A$ is indeed an achronal boundary.

Now, assume that $(M,g)$ is globally hyperbolic. By Proposition \ref{first}(i), the leaf space $Q$ is a smooth $n$-manifold $Q$, and the standard projection $\pi: M \rightarrow Q$ is a smooth onto submersion. Thus, $\pi \circ \zeta \circ i$ is a homeomorphism between $A$ and $Q$, where $i: x \in A \hookrightarrow (0,x) \in \mathbb{R}\times A$. Since $Q$ is homeomorphic to any Cauchy hypersurface, then so is $A$.

Finally, suppose $(M,g)$ is globally hyperbolic and $A$ is compact. Let $\alpha: (a,b) \subset \mathbb{R} \rightarrow M$ be any future-directed, inextendible timelike curve in $(M,g)$ ($-\infty \leq a< b \leq +\infty$). Then $\zeta ^{-1} \circ \alpha: \lambda \in (a,b) \subset \mathbb{R} \mapsto (t(\lambda),x(\lambda)) \in \mathbb{R}\times A$ is continuous, and the fact that $\alpha$ is future-directed implies that $t$ is an increasing function. Fix any $\lambda_0 \in (a,b)$, and suppose $\alpha$ does not intersect $A$. In this case, $t(\lambda)$ is never zero, and we may assume, say, that $t(\lambda) <0$ for all $\lambda \in (a,b)$, since the case if it always $>0$ is entirely analogous. But then (since the $t$-coordinate increases) the curve $\zeta ^{-1} \circ \alpha |_{[\lambda _0, b)}$ stays imprisoned in the compact set $[t(\lambda_0),0]\times A$, and hence the future-inextendible timelike curve $\alpha |_{[\lambda _0, b)}$ stays imprisoned in the compact set $\zeta([t(\lambda_0),0]\times A) \subset M$, which contradicts the strong causality of $(M,g)$ (see for instance Lemma 13, Ch. 14 in \cite{oneill}). We conclude that $\alpha$ intersects $A$. We know it does so exactly once by the achronality of $A$, so $A$ is indeed a Cauchy hypersurface.
\end{proof}

Another relevant fact, which is easy to check, is that {\em if $X$ is a timelike conformal Killing vector field in $(M,g)$, then $X$ is timelike Killing for $(M,\tilde{g})$}, where
 \begin{equation}
 \label{starmetric}
 \tilde{g} :=  (1/\beta^2)g, \;\; \beta := \sqrt{-g(X,X)}.
 \end{equation}

The following characterization will be technically useful later.
\begin{prop}
\label{second}
Assume that $(M,g)$ is globally hyperbolic with compact Cauchy hypersurfaces, and let $X \in \mathfrak{X}(M)$ be a timelike conformal Killing vector field in $(M,g)$. Then, using the notation in (\ref{starmetric}), we have that $X$ is complete if and only if $(M,\tilde{g})$ is geodesically complete.
\end{prop}
\begin{proof}
We know that $X$ is a timelike Killing vector field for $(M,\tilde{g})$. If the latter is geodesically complete, then $X$ is complete by Prop. 30, Ch.9, p. 254 of \cite{oneill}.

Conversely, assume that $X$ is complete. Then, the flow $\phi:\mathbb{R}\times M \rightarrow M$ defines (cf. the proof and notation of Proposition \ref{first} (i)) is a free proper $\mathbb{R}$-action on $M$ which gives $\pi:M \rightarrow Q$ the structure of a principal $\mathbb{R}$-bundle. Since $X$ is unit Killing for $\tilde{g}$, we have, by Proposition \ref{first}(ii), that it has the form
\[
\tilde{g} = g = -\theta\otimes \theta + \pi^{\ast} h,
\]
but now $\theta (=\theta_X)$ can clearly be viewed as a connection $1$-form over the bundle $\pi:M \rightarrow Q$, and hence this spacetime is precisely of the form considered in Example 2.4 in \cite{RS} (with $f\equiv 1$). Therefore it is geodesically complete by Proposition 2.1 in that reference.
%Conversely, assume that $X$ is complete. From Prop. \ref{firstt} the spacetime is precisely of the form considered in Example 2.4 in \cite{RS} (with $f\equiv 1$). Therefore it is geodesically complete by Proposition 2.1 in that reference.
\end{proof}

%the flow $\phi:\mathbb{R}\times M \rightarrow M$ defines (cf. the proof and notation of Proposition \ref{first} (i)) is a free proper $\mathbb{R}$-action on $M$ which gives $\pi:M \rightarrow Q$ the structure of a principal $\mathbb{R}$-bundle. Since $X$ is unit Killing for $\tilde{g}$, we have, by Proposition \ref{first}(ii), that it has the form
%\[
%\tilde{g} = g = -\theta\otimes \theta + \pi^{\ast} h,
%\]
%but now $\theta (=\theta_X)$ can clearly be viewed as a connection $1$-form over the bundle $\pi:M \rightarrow Q$, and hence this spacetime is precisely of the form considered in Example 2.4 in \cite{RS} (with $f\equiv 1$). Therefore it is geodesically complete by Proposition 2.1 in that reference.
%\end{proof}

\section{Proof of Theorem \ref{mainthm1}}\label{S2}

%...Supongo que esta prueba se suprime, no?...............
%
%Let $\alpha$ be any inextendible timelike curve and let $P = I^-(\alpha)$. We wish to show that $P=M$. Assume, by way of contradiction, that this is not the case, so that the boundary $\partial P$ is non-empty. $\partial P$ is an achronal boundary, and hence achronal and edgeless, and thus compact by Proposition \ref{jsmain1}.
%
%Now, $B:= \partial P \setminus \overline{Im \alpha}$, if non-empty, is a union of future-directed null geodesics (see, e.g., Thm. 8.1.6 in \cite{wald}) which either are future-inextendible or else have endpoints on $\overline{Im \alpha}$.
%However, since $(M,g)$ is in particular strongly causal, $\alpha$ has closed image (see, e.g.. exercise 11, p. 438 of \cite{oneill}), and therefore $\overline{Im \alpha} \equiv Im \alpha$, whence we conclude that $\partial P \cap \overline{Im \alpha} \equiv \emptyset$ since $Im\alpha \subset P$. Thus, $B=\partial P$. If a null geodesic generator of $\partial P$ had an endpoint on $Im \alpha$, we would have that $\partial P \cap I^{-}(\alpha) \neq \emptyset$, which is impossible. Thus, $\partial P$ is a union of null geodesics which are all future-inextendible, and these cannot be contained in a compact set (again by strong causality), which gives the desired contradiction. Hence, $I^{-}(\alpha) = M$.
%
%By an entirely analogous argument, we establish that $I^{+}(\alpha) = M$, so the proof is complete.
%
%
%\qcd

In view of Proposition \ref{firstt}, and up to a conformal factor (which does not affect the purely causal arguments needed here) we can suppose
\[
M={\mathbb R}\times M_0,\quad g=-dt^2+2\omega(\cdot)dt+g_0,
\]
with $M_0$ compact. Let $\alpha$ be any inextendible timelike curve and let $P = I^-(\alpha)$. We wish to show that $P=M$. Pick an arbitrary point $(t_*,x_*)\in M$. There is no loss of generality in assuming that the domain of $\alpha$ is $[0,+\infty)$ and we do so for the rest of the argument; we wish to show that $(t_*,x_*)\ll \alpha(s_*)=(t(s_*),x(s_*))$ for a large enough $s_* \in [0, +\infty)$. Consider the curve $\gamma(s)=(\tau(s),y(s))$ defined on the interval $[0,1]$, where $\tau(s):=(t(s_*)-t_*)s+t_*$, and $y(s)$ is chosen to be any minimizing $g_0$-geodesic with $y(0)=x_*$, $y(1)=x(s_*)$. Clearly, $\gamma(0)=(t_*,x_*)$ and $\gamma(1)=(t(s_*),x(s_*))$. It suffices to show that $\gamma$ is timelike for $s_*$ big enough. Since $M_0$ is compact, and $\omega$, $g_0$ are independent of the $t$-coordinate, we have, for large enough $s_*$,
\[
\dot{\tau}(s)=t(s_*)-t_*>\omega(\dot{y}(s))+\sqrt{\omega(\dot{y}(s))^2+g_0(\dot{y}(s),\dot{y}(s))}\quad \hbox{for all $s\in [0,1]$.}
\]
Hence, $g(\dot{\gamma}(s),\dot{\gamma}(s))=-\dot{\tau}(s)^2+2\omega(\dot{y}(s))\dot{\tau}(s)+g_0(\dot{y}(s),\dot{y}(s))<0$ for all $s\in [0,1]$, as required. \qcd

{\em Proof of Corollary \ref{maincor1}.}
As mentioned in the Introduction, the implication $(ii) \Longrightarrow (iii)$ follows immediately from Lemma 1 of \cite{GH}, so the chain of implications
\[
(i) \Longrightarrow (ii) \Longrightarrow (iii) \Longrightarrow (iv)
\]
is clear. The final implication $(iv) \Longrightarrow (i)$ follows from Theorem \ref{mainthm1} together with Galloway's result \cite{G1} mentioned in the Introduction.
\qcd

\section{Proof of Theorem \ref{mainthm2}}\label{S3}

To fix ideas, we may always assume that $X$ is future-directed. Let $\phi: U\subset \mathbb{R}\times M \rightarrow M$ denote the global flow of $X$, and fix a smooth spacelike Cauchy hypersuface $S$ in $(M,g)$. We know that $U$ contains $\{0\}\times M$, is open in the product topology in $\mathbb{R}\times M$, and it equals the latter set iff $X$ is complete. Let $U_S:= (\mathbb{R}\times S)\cap U$ and onsider the smooth map $\phi_S:= \phi|_{U_S}$. The achronality of $S$ implies that $\phi_S$ is one-to-one, and hence a smooth homeomorphism onto an open set of $M$ by Invariance of Domain. Since $X$ is in particular non-zero everywhere, $\phi_S$ is actually a local diffeomorphism (by the Inverse Function Theorem) and hence a diffeomorphism onto its image. Finally, the fact that $S$ is Cauchy implies that $\phi_S$ is onto. We conclude that the map $\phi_S: U_S \rightarrow M$ is a global diffeomorphism. In particular, $U_S$ is an open set in $\mathbb{R}\times S$ diffeomorphic to $\mathbb{R}\times S$ itself.

We denote generic points of $U_S$ by $(t,x)$, with $t$ viewed as a time coordinate and $x \in S$. We also denote by $\partial_t$ the lift, by the projection $\pi_1:(t,x) \in \mathbb{R}\times S \mapsto t \in \mathbb{R}$, of the standard vector field $d/dt$ on $\mathbb{R}$, and $dt := d\pi_1$, so that
\begin{equation}
\label{dete}
dt(\partial_t) \equiv 1.
\end{equation}
We also keep the notation $\partial_t$ for the restriction of this vector field to $U_S$ if there is no risk of confusion. Accordingly, on $U_S$ we have by construction
\begin{equation}
\label{eq1}
(\phi_S)_{\ast}(\partial_t) = X.
\end{equation}

As mentioned in section \ref{S1}, $X$ is a unit Killing vector field with respect to the metric $\tilde{g}$ given in (\ref{starmetric}). We shall analyze the globally hyperbolic spacetime $(M, \tilde{g})$ more closely.

Using the concepts and notation of Proposition \ref{first}, we have a semi-Riemannian submersion $\pi:(M,\tilde{g}) \rightarrow (Q,h)$, for some (uniquely given) Riemannian metric $h$ on the quotient space $Q=Q_X$, and (cf. Eq. (\ref{stationarymetric}))
\[
\tilde{g} = -\theta \otimes \theta + \pi^{\ast}h,
\]
where
\[
\theta = - \tilde{g}(X, \, . \, ) \equiv - g(X/\beta, \, . \,),
\]
with $\beta = \sqrt{-g(X,X)}$. Consider the pullback metric
\begin{equation}
\label{pullback1}
g_S := \phi _S^{\ast} \tilde{g} \equiv = -(\phi _S^{\ast}\theta)\otimes (\phi _S^{\ast}\theta) + (\pi\circ \phi_S)^{\ast}(h)
\end{equation}
on $U_S$. Then $\phi_S: (U_S,g_S) \rightarrow (M,\tilde{g})$ becomes an isometry by construction. In particular, note that (\ref{eq1}) implies that $\partial_t$ is a unit Killing vector field in $(U_S,g_S)$. Moreover, since isometries take (spacelike) Cauchy hypersurfaces onto (spacelike) Cauchy hypersurfaces, $(U_S,g_S)$ is globally hyperbolic and the hypersurfaces of the form
\[
\{t\}\times S
\]
which are contained in $U_S$ are then spacelike Cauchy hypersurfaces therein.

We wish to show now that $g_S$ will have the general form (\ref{standardstationary}) of a standard stationary metric. In order to do that, we define on $U_S$ the 1-form $\omega \in \Omega^1(U_S)$ given by
\[
\omega := dt - \phi_S^{\ast} \theta.
\]
Now, using (\ref{eq1}) and the definition of $\theta$, we conclude
\[
(\phi_S^{\ast}\theta)(\partial_t) = \theta((\phi_S)_{\ast}(\partial_t)) = \theta(X) \equiv 1.
\]
Therefore,
\[
\omega(\partial_t) \equiv 0,
\]
and clearly, ${\cal L}_X \theta =0$, so that
\[
{\cal L}_{\partial_t}\omega =0.
\]
We conclude that there exists a unique 1-form $\omega_0 \in \Omega^1(S)$ for which
\begin{equation}
\label{eq2}
\omega = \pi_2^{\ast}\omega_0,
\end{equation}
where $\pi_2: \mathbb{R}\times S \rightarrow S$ is the canonical projection onto the second factor.

To proceed, consider the mapping $m:S \rightarrow Q$ given by
\[
m: x \in S \mapsto \pi \circ \phi_S(0,x) \in Q.
\]
It not hard to check $m$ is a smooth diffeomorphism. Use it to define a Riemannian metric $h_0$ on $S$ by the pullback:
\[
h_0:= m^{\ast} h.
\]
Finally, note that
\[
m \circ \pi _2(t,x) = \pi (\phi_S(0,x)) \equiv \pi(\phi_S(t,x)), \forall (t,x) \in U_S,
\]
i.e. $m \circ \pi _2 |_{U_S}\equiv \pi \circ \phi_S$, whence we conclude that
\begin{equation}
\label{eq3}
(\pi \circ \phi_S)^{\ast}h = \pi _2^{\ast} h_0.
\end{equation}
Substituting (\ref{eq2}) and (\ref{eq3}) in (\ref{pullback1}) and rearranging, we get
\begin{equation}
\label{eq4}
g_S = -dt\otimes dt - \pi_2^{\ast} \omega_0 \otimes dt - dt \otimes \pi_2^{\ast} \omega_0 + \pi_2^{\ast}g_0,
\end{equation}
where we have defined
\begin{equation}
\label{eq5}
g_0:= h_0 - \omega_0 \otimes \omega_0.
\end{equation}
Note that the smooth $(0,2)$-tensor $g_0$ is actually the induced metric on each $t=const.$ hypersuface in $U_S$, and hence positive-definite. (\ref{eq4}) is the desired standard stationary form.

The next step is now clear: we {\em define} $(\hat{M}, \hat{g})$ as the standard stationary spacetime (cf. Example \ref{ex1}) associated with the data $(S,g_0,\beta_0 \equiv 1, \omega_0)$, which has precisely the form (\ref{eq4}), and a complete unit Killing vector field $\hat{X} \equiv \partial_t$ (cf. item (1) in Example \ref{ex1}). Clearly, this spacetime is an isometric (trivial) extension of $(U_S,g_S)$, and the map
\[
\varphi: \phi_S ^{-1}
\]
gives the desired conformal embedding.

To complete the proof, assume that $S$ is compact. We wish to show that $\varphi(S) \equiv \{0\} \times S$ is a Cauchy hypersurface for $(\hat{M}, \hat{g})$. Let $\alpha: \lambda \in (a,b) \subset \mathbb{R} \mapsto (t(\lambda),x(\lambda)) \in \mathbb{R}\times S$ be any future-directed causal inextendible curve in $(\hat{M}, \hat{g})$ ($-\infty \leq a< b \leq +\infty$). In particular, the fact that this curve is future-directed implies that we have $\dot{t}(\lambda) >0$, where the dot indicates derivative with respect to the curve parameter. Fix any $\lambda_0 \in (a,b)$, and suppose $\alpha$ does not intersect $\varphi(S)$. In this case, we may assume, say, that $t(\lambda) <0$ for all $\lambda \in I$, since the case if it always $>0$ is entirely analogous. But then (since the $t$-coordinate increases) the future-inextendible curve $\alpha |_{[\lambda _0, b)}$ stays imprisoned in the compact set $[t(\lambda_0),0]\times S$, which contradicts the strong causality of  $(\hat{M}, \hat{g})$ (cf. (2) in Example \ref{ex1}).  We conclude that $\alpha$ does intersect $\varphi(S)$, and since $\alpha$ is arbitrary, and $\varphi(S)$ is acausal (again by (2) in Example \ref{ex1}), it is indeed a Cauchy hypersurface as claimed.
\qcd

\section{Proof of Theorem \ref{mainthm4}}\label{S4}

%As remarked in the Introduction, the Lorentzian splitting theorem ensures that Bartnik's conjecture is true provided one can show, in the presence of its hypotheses, the existence of a {\em geodesic timelike line}. On the other hand, it is easy to show that if $(M,g)$ is globally hyperbolic with compact Cauchy hypersurfaces, then {\em either} a null {\em or} a timelike geodesic line {\em does} exist.

Our goal in this section is to understand in a more detailed fashion what can be said about the structure of the spacetime in the context of the Bartnik conjecture when one has an {\em incomplete}  conformal Killing vector field defined thereon.

%It is therefore of interest to explore the consequences for such a spacetime of the existence of a geodesic {\em null} line.

%This problem has been extensively explored by G. Galloway and collaborators \cite{1,2,3}, culminated by the so-called {\em null splitting theorem} in \cite{GNull}. This has been applied in D. Solis thesis \cite{didier} to obtain rigidity results for spacetimes with a null line. In this section, we adapt some of the results in \cite{didier} to obtain new rigidity results for spacetimes with a null geodesic line.

We again assume, throughout this section, that $(M,g)$ is globally hyperbolic with compact Cauchy hypersurfaces and admits a timelike conformal Killing vector field $X \in \mathfrak{X}(M)$. No a priori assumption on the Ricci tensor is made.  We shall make extensive use of the conformally extended spacetime $(\hat{M},\hat{g})$ wherein $X$ extends to a complete unit timelike Killing vector field $\hat{X}$, as described in Theorem \ref{mainthm2}. We denote the boundary of $M$ when viewed as an open subset of $\hat{M}$ by $\partial _{\hat{M}}M \equiv \partial M$.

The proof of Theorem \ref{mainthm2} makes it clear that a point $p \in \hat{M}$ is on $\partial M$ if and only if it is an endpoint in $\hat{M}$ of an {\em incomplete} maximal integral curve of $X$. In particular, $\partial M =\emptyset$ if and only if $X$ is itself complete in $M$. We start our discussion by examining more closely some facts about $M$ and $\partial M$.

\begin{lemma}
\label{lemma1}
$M \subset \hat{M}$ is {\em causally convex}, i.e., any causal curve segment in $(\hat{M},\hat{g})$ with endpoints in $M$ is contained in $M$.
\end{lemma}
\begin{proof}
Let $\alpha:[0,1] \rightarrow \hat{M}$ be a future-directed causal curve segment in $(\hat{M},\hat{g})$ with $\alpha(0),\alpha(1) \in M$, and $\alpha(t_0) \notin M$ for some $0<t_0<1$. Let $S$ be a smooth spacelike (hence acausal) Cauchy hypersurface of $(M,g)$ with $\alpha(1) \in S$. Since $S$ is also an acausal Cauchy hypersurface of $(\hat{M},\hat{g})$, $\alpha(t) \notin S$ for any $t \in [0,1)$. In particular, $\alpha(0) \in I^+(S,M)\cup I^-(S,M)$. However, if $\alpha(0) \in I^+(S,M) \subset I^+(S,\hat{M})$, then for some $x \in S$, we would have
\[
x \ll_{(\hat{M},\hat{g})} \alpha(0) \leq _{(\hat{M},\hat{g})}\alpha(1) \Longrightarrow x \ll_{(\hat{M},\hat{g})}\alpha(1),
\]
which violates the achronality of $S$ in $(\hat{M},\hat{g})$. We conclude that $\alpha(0) \in I^-(S,M)$.

Now, let
\[
s_0 := \sup \{ t \in [0,1] \, : \, \alpha[0,t] \subset M\}.
\]
Then, $s_0 \leq t_0 <1$ and hence $\alpha|_{[0,s_0)}$ is a future-inextendible causal curve in $(M,g)$ starting at $\alpha(0)$ which does not intersect $S$, an absurd since $S$ is Cauchy.
\end{proof}

\begin{lemma}
\label{lemma2}
If $S,S'\subset M$ are two smooth spacelike Cauchy hypersurfaces in $(M,g)$, then
\begin{equation}
\label{eqfor2}
I^\pm(S,\hat{M})\cap \partial M = I^\pm(S',\hat{M})\cap \partial M.
\end{equation}
\end{lemma}
\begin{proof}
We only need to show that
\[
I^+(S,\hat{M})\cap \partial M \subset I^+(S',\hat{M})\cap \partial M,
\]
since then the opposite inclusion is immediate by exchanging the roles of $S$ and $S'$, and the past case proceeds by time-dual arguments.

But note that $S'$ being also Cauchy in $(\hat{M}, \hat{g})$ means that any putative $p \in I^+(S,\hat{M})\cap \partial M \setminus I^+(S',\hat{M})$ would be in $J^-(S',\hat{M})$, and hence we might juxtapose a future-directed timelike curve from $S$ to $p$ with a future-directed causal curve from $p$ to $S'$, both in $(\hat{M}, \hat{g})$. However, since the resulting composition of these two curves could be viewed as a causal curve in $(\hat{M}, \hat{g})$  with endpoints in $M$ which leaves $M$, this would contradict \ref{lemma1}, and hence the desired inclusion must hold.
\end{proof}

With Lemma \ref{lemma2} in mind, we may define, for some smooth spacelike Cauchy hypersurface $S \subset M$,
\[
\partial _{\pm} M := I^{\pm} (S,\hat{M})\cap \partial M,
\]
and it is thus guaranteed that this definition does not depend on the particular choice of $S$. Yet, the achronality of $S$ implies that $\partial _{+} M \cap \partial _{-} M = \emptyset$, while the fact that $S$ is Cauchy for $(\hat{M}, \hat{g})$ ensures that $\partial M = \partial _{+} M \cup \partial _{-} M$.

The following result summarizes the structural properties of the partial boundaries $\partial _{+} M$ and $\partial _{-} M$.
\begin{prop}
\label{lemma3}
The sets $\partial _{+} M$ and $\partial _{-} M$ are edgeless achronal sets in $(\hat{M},\hat{g})$. Indeed, if $\partial_{\pm}M$ is non-empty, it is a Cauchy hypersurface in $(\hat{M},\hat{g})$, and hence homeomorphic to any given Cauchy hypersurface of $(M,g)$. Moreover, if both $\partial _{+} M$ and $\partial _{-} M$ are both non-empty, then
\begin{equation}
\label{equality}
M = I^-(\partial _{+} M, \hat{M}) \cap I^+(\partial _{-} M, \hat{M}).
\end{equation}
In particular, $\overline{M}$ is compact in this case.
\end{prop}
\begin{proof}
To show that $\partial_{\pm}M$ is edgeless and achronal in $(\hat{M},\hat{g})$, we again need to argue only for, say, $\partial _{+} M$, the other case following from time duality.

Suppose, by way of contradiction, that there exist $p, q \in \partial _{+} M$ such that $p \ll_{(\hat{M}, \hat{g})} q$. From the definition of  $\partial _{+} M$ together with the openness of the chronological relation, we can pick $p',q' \in \hat{M}$ and $S \subset M$ smooth spacelike Cauchy hypersurface such that (i) $p' \in I^{+} (S,\hat{M})\setminus M$ and (ii) $q' \in I^{+} (p',\hat{M})\cap M$. Again by suitably juxtaposing causal curves we can obtain a causal curve in $(\hat{M}, \hat{g})$ with endpoints in $M$ but leaving $M$, in contradiction with Lemma \ref{lemma1}. We conclude that $\partial _{+} M$ is indeed achronal in $(\hat{M},\hat{g})$.

Now, we show that $edge(\partial_{+}M)\equiv \emptyset$. Suppose, again by way of contradiction, that $p \in edge(\partial_{+}M)$. Fix some smooth spacelike Cauchy hypersurface $S\subset M$ for which $p \in I^+(S,\hat{M})=:U$. Since $U$ is open in $\hat{M}$ and $p$ is an edge point, there exists a timelike curve segment $\beta:[0,1] \rightarrow U$ starting at $I^-(p,U)$ and ending at $I^+(p,U)$ which does not cross $\partial_{+}M$. Note, however, that such a curve is entirely contained in $U$. We claim that $\beta(0) \in M$ and $\beta(1) \notin M$. Note that in this case $\beta$ would intersect $\partial M \cap U \equiv \partial _+ M$, yielding the desired contradiction.

But if $\beta(0) \notin M$, since $q \ll_{(\hat{M},\hat{g})}  \beta(0) \ll_{(\hat{M},\hat{g})} p$ for some $q \in M$ and $p \in \partial M$, then $q \ll_{(\hat{M},\hat{g})}  \beta(0) \ll_{(\hat{M},\hat{g})} p'$ for some $p'\in M$, which violates Lemma \ref{lemma1}. Thus, $\beta(0) \in M$. An entirely analogous reasoning establishes that $\beta(1) \notin M$. This finishes the proof that $\partial _{+} M$ is edgeless.

Therefore, if $\partial _{\pm} M \neq \emptyset$, Proposition \ref{jsmain1} establishes it is a Cauchy hypersurface in $(\hat{M},\hat{g})$, since the Killing vector field $\hat{X}$ is complete.

Finally, suppose $\partial _{\pm} M$ are both non-empty. It has been established that these are Cauchy hypersurfaces in $(\hat{M},\hat{g})$, and the arguments in the previous paragraphs actually serve to prove that points to the future of $\partial_+M$ and to the past of $\partial_- M$, respectively, cannot be in $M$. Therefore,
\[
M \subset I^-(\partial _{+} M, \hat{M}) \cap I^+(\partial _{-} M, \hat{M}).
\]
Now, let $p \in I^-(\partial _{+} M, \hat{M}) \cap I^+(\partial _{-} M, \hat{M})$, and consider a future-directed timelike curve segment $\gamma$ of $(\hat{M},\hat{g})$ with endpoints on $\partial_+M$ and $\partial_- M$ and passing through $p$. If $p \notin M$, then ``perturbing'' the endpoints of $\gamma$ a little, one might get another timelike curve with endpoints in $M$ passing through $p$, or in other words leaving $M$, again contradicting Lemma \ref{lemma1}. So $p \in M$, and the equality (\ref{equality}) is established. Moreover, the global hyperbolicity of $(\hat{M},\hat{g})$ implies that the right hand side of (\ref{equality}) is precompact since $\partial_+M$ and $\partial_- M$ are compact, thus showing that $\overline{M}$ is compact in $\hat{M}$.
\end{proof}

\begin{prop}
\label{line1}
Any future-directed null geodesic line in $(M,g)$ has a future endpoint on $\partial _{+} M$ and a past endpoint on $\partial _{-} M$. In particular, if such a null line exists, then $\partial_+M$ and $\partial_- M$ are both non-empty, $\overline{M}$ is compact in $\hat{M}$, and {\em every} integral curve of $X$ is incomplete both to the past and to the future.
\end{prop}
\begin{proof} Let $\eta: (a,b) \rightarrow M$ be a future-directed null geodesic line ($-\infty \leq a <b \leq +\infty$), and pick any $t_0 \in (a,b)$. Suppose that $\eta_0:=\eta|_{[t_0, b)}$ is future-inextendible in $(\hat{M}, \hat{g})$. Since $\hat{X}$ is a complete timelike Killing vector field, we may use \ref{mainthm1} to conclude that the NOH holds in $(\hat{M}, \hat{g})$. Therefore $I^-(\eta_0) \equiv \hat{M}$. But then there exists $s_0 \in (t_0, b)$ for which $\eta(t_0) \ll_{(\hat{M}, \hat{g})} \eta(s_0)$, and in view of Lemma \ref{lemma1} this implies that $\eta(t_0) \ll_{(M, g)} \eta(s_0)$, contradicting the maximality of the null segment $\eta_{[t_0,s_0]}$.

We therefore conclude that $\eta_0$, and hence $\eta$, cannot be future-inextendible in $(\hat{M}, \hat{g})$, and thus must have an endpoint therein, which in addition is clearly on $\partial M$, say $p \in \partial M$. Given an smooth spacelike Cauchy hypersurface $S \subset M$, since $S$ is acausal in $(\hat{M}, \hat{g})$, $\eta_0$ must enter in $I^{+} (S,\hat{M})$ and remain there, so $p \in I^{+} (S,\hat{M})\cap \hat{M} \equiv \partial _+M$ as desired. An analogous, time-dual argument establishes that $\eta$ has a past endpoint on $\partial _{-} M$.

To complete the proof, note that Proposition \ref{lemma3} shows that $\partial _{+} M$ and $\partial _{-} M$ are Cauchy hypersurfaces in $(\hat{M}, \hat{g})$, and hence every integral curve of $\hat{X}$ will intersect both exactly once at finite value of its parameters, which yields the incompleteness of each integral curve of $X$.
\end{proof}

\begin{cor}
\label{bartnikpartial1}
Assume that the following holds for the globally hyperbolic spacetime $(M,g)$:
\begin{itemize}
\item[i)] the TCC is satisfied and $(M,g)$ has compact Cauchy hypersurfaces (so that $(M,g)$ is cosmological);
\item[ii)] $(M,g)$ is timelike geodesically complete.
\end{itemize}
Then, either $(M,g)$ splits as in the Bartnik conjecture, or else
\begin{enumerate}
\item any future-directed integral curve of $X$ has a future endpoint on $\partial _{+} M$ and a past endpoint on $\partial _{-} M$. In particular, $\partial_+M$ and $\partial_- M$ are both non-empty, and
\item $\overline{M}$ is compact in $\hat{M}$, and {\em every} integral curve of $X$ is incomplete both to the past and to the future.
\end{enumerate}
\end{cor}
\begin{proof} Indeed, we know that $(M,g)$ admits either a null or a timelike line. But under assumptions $(i)-(ii)$, if the latter occurs then $(M,g)$ splits by the Lorentzian splitting theorem and we are done. We can therefore assume that $(M,g)$ admits a null line, in which case the conclusions follow from Proposition \ref{line1}.
\end{proof}

\begin{cor}
\label{bartnikpartial2}
Assume that the conditions $(i)-(ii)$ of Corollary \ref{bartnikpartial1} hold for the globally hyperbolic spacetime $(M,g)$. If $(M,g)$ does not split, then the norm $\beta:= \sqrt{- g(X,X)}$ of the conformal Killing vector field $X$ is unbounded along every future- or past-inextendible timelike geodesic.
\end{cor}
\begin{proof} Since $(\hat{M},\hat{g})$ is in particular stably causal, we may choose a temporal function, i.e., a smooth function $f: \hat{M} \rightarrow \mathbb{R}$ with past-directed timelike gradient $\hat{\nabla } f$. For the sake of clarity, we adopt for the rest of this proof the following notation. Given $v \in TM$ its norm with respect to $g$ and $\hat{g}$ will be denoted by $|v|$ and $\| v\|$, respectively. In particular, if $v \in T_pM$, then
\begin{equation}
\label{eq0aux}
\| v \| \equiv |v|/\beta(p).
\end{equation}

As long as we are assuming here that $(M,g)$ does not split, Corollary \ref{bartnikpartial1} implies in particular that $\overline{M}$ is compact in $\hat{M}$, so $f$ is bounded in $M$, and we can also pick a constant $c>0$ for which
\begin{equation}
\label{eq1aux}
\| \hat{\nabla } f \| \geq c \mbox{ in $M$}.
\end{equation}

Consider a future-inextendible (and thus future-complete) timelike $g$-geodesic $\gamma: [0,+\infty) \rightarrow M$. If $\beta \circ \gamma $ were bounded, then
\begin{equation}
\label{eq2aux}
\frac{1}{\beta \circ \gamma } \geq A \mbox{ for some positive number $A$}.
\end{equation}
But then, we have, for each $t \in [0,+\infty)$,
\begin{eqnarray}
f(\gamma(t)) - f(\gamma(0)) &=&\int_0^t (f \circ \gamma)'(s) ds = \int_0^t \hat{g}(\hat{\nabla } f (\gamma(s)), \gamma '(s)) ds \\ &\geq &  \int_0^t \| \hat{\nabla } f (\gamma(s))\|\| \gamma'(s)\| ds \geq cA  \int_0^t | \gamma'(s)| ds,
\end{eqnarray}
where we have used the reverse Cauchy inequality for timelike vectors for the first inequality on the right-hand side and the bounds in (\ref{eq1aux}) and (\ref{eq2aux}), as well as the relationship in (\ref{eq0aux}), for the second inequality. However, since the last term on the last equation diverges as $t$ goes to infinity since $\gamma$ is future-complete, we would conclude that $f$ is unbounded in $M$, a contradiction.
\end{proof}

We finally have

{\em Proof of Theorem \ref{mainthm4}.}
Assume, to the contrary, that $(M,g)$ does not split. Then, due to Corollary \ref{maincor1}, we may assume that $X$ is incomplete. \cambios{Thus, under our assumptions, either alternative $(iii.2)$ or alternative $(iii.3)$ in the statement of Theorem \ref{mainthm4} hold.}

Let $\gamma :[0,+\infty) \rightarrow M$ be a future-complete unit timelike $g$-geodesic. Then its $\hat{g}$-arc length reparametrization will be denoted by $\hat{\gamma}:[0,\hat{\ell}) \rightarrow M$.

 The first key point to note is that, due to Corollary \ref{bartnikpartial1}, $\overline{M}$ is compact in $\hat{M}$, and hence $\hat{\gamma}$ cannot be future-inextendible in $(\hat{M},\hat{g})$; consequently $\hat{\ell} < \infty $.

\cambios{Now, if $(iii.3)$ holds, then $\sigma \circ \hat{\gamma}$ is bounded. However, since  $\sigma \equiv X(\log \beta)$ (cf. Eq. (\ref{conformalfunction})), this would imply that so is $\beta \circ \hat{\gamma}$, which contradicts Corollary \ref{bartnikpartial2}.

Therefore, we shall assume, for the rest of the proof, that alternative $(iii.2)$ holds.}

The next ingredient is the well-known formula relating the Ricci tensor of conformally rescaled metrics (see, e.g., p. 59 of \cite{besse})
\begin{eqnarray}
\label{rescaling1}
Ric &=& \hat{Ric} + (1-n) \left[ \hat{Hess} _{\log \beta} - d \log \beta \otimes d \log \beta \right] \nonumber \\
&-& \left[ \hat{\triangle }\log \beta + (n-1) \hat{g}(\hat{\nabla }\log \beta, \hat{\nabla }\log \beta \right] \hat{g}.
\end{eqnarray}
After a straightforward calculation with formula (\ref{rescaling1}), we get, on the one hand,
\begin{equation}
\label{rescaling2}
Ric(\hat{\gamma}',\hat{\gamma}') = \hat{Ric}(\hat{\gamma}',\hat{\gamma}') - (n-1)(\log \beta \circ \hat{\gamma})'' + (\hat{\triangle }\log \beta)\circ \hat{\gamma},
\end{equation}
and on the other hand, tracing (\ref{rescaling1}) we get, after a few rearrangements,
\begin{equation}
\label{rescaling3}
\hat{\triangle }\log \beta = \frac{\hat{R} - \beta ^2 R}{2n} - \frac{(n-1)}{2} \hat{g}(\hat{\nabla }\log \beta, \hat{\nabla }\log \beta) .
\end{equation}

An easy computation also reveals that
\begin{equation}
\label{gradientbeta}
\hat{\nabla} \log \beta = \beta^2 \nabla \log \beta = - \beta\sigma U + \beta^2 \nabla _U U,
\end{equation}
where $U:= X/\beta $. Gathering together Eqs. (\ref{rescaling2}), (\ref{rescaling3}) and (\ref{gradientbeta}) we end up with
\begin{eqnarray}
\label{rescaling4}
Ric(\hat{\gamma}',\hat{\gamma}') &=& \hat{Ric}(\hat{\gamma}',\hat{\gamma}') - (n-1)(\log \beta \circ \hat{\gamma})''  + \frac{\hat{R}\circ \hat{\gamma} }{2n} %\nonumber \\
 \\ &-&\frac{(n-1)}{2} (\beta \circ \hat{\gamma}) ^2 g(\nabla _U U,\nabla _U U)\circ \hat{\gamma}+ \frac{(n-1)}{2}\left[\sigma ^2 - \frac{\beta ^2 R}{n(n-1)}\right] \circ \hat{\gamma}.\nonumber
\end{eqnarray}
Finally, about Eq. (\ref{rescaling4}), we note the following points.
\begin{itemize}
\item[1)]
 $\hat{R}\circ \hat{\gamma}$ is bounded, as per the first key point above;
\item[2)] $g(\nabla _U U,\nabla _U U) \geq 0$, since $U$ is a $g$-unit timelike field, and hence $g(\nabla _U U,U)\equiv 0$;
\item[3)] $(\log \beta \circ \hat{\gamma})''$ is not bounded above. Indeed, $log\beta \circ \hat{\gamma}$ and $(log\beta \circ \hat{\gamma})'$  are defined over a finite interval (recall that $\hat{\ell}<\infty$), and hence it is unbounded due to Corollary \ref{bartnikpartial2}.
\end{itemize}
We now argue that the first term on the right-hand side of Eq. (\ref{rescaling2}), namely  $\hat{Ric}(\hat{\gamma}',\hat{\gamma}')$, is also bounded in $M$. The upper bound (\ref{bound}) on the last term of Eq. (\ref{rescaling4}) and the TCC condition taken together imply that the following inequality holds:
\begin{equation}\label{hhh}
(n-1)(\log\beta\circ\hat{\gamma})''\leq\hat{Ric}(\hat{\gamma}',\hat{\gamma}')+D\quad\hbox{for some $D>0$.}
\end{equation}
On the other hand, given a Cauchy surface $S$, the spacetime $(\hat{M},\hat{g})$ can be seen as the standard stationary spacetime (cf. Example \ref{ex1}) associated with the data $(S,g_0,\beta_0 \equiv 1, \omega_0)$ (recall the proof of Theorem \ref{mainthm2}). So, taking into account that $\hat{\gamma}\equiv (\hat{t},\hat{x})$ is timelike, we deduce the existence of some $\mu>0$ such that
\[
\hat{t}'^2\geq \mu\; g_0(\hat{x}',\hat{x}')=\mu\,\|\hat{x}'\|_0^2.
\]
So, if we define the Riemannian metric $g_R$ on $\hat{M}$ as follows,
\[
g_R:=dt^2+g_0,
\]
then
\begin{equation}\label{h}
\|\hat{\gamma}'\|_R^2=\hat{t}'^2+\|\hat{x}'\|_0^2\leq \left(\frac{1}{\mu}+1\right)\hat{t}'^2.
\end{equation}
On the other hand, the compactness of $\overline{M}$ in $\hat{M}$ gives
\begin{equation}\label{hh}
\hat{Ric}(\hat{\gamma}',\hat{\gamma}')\leq C\,\|\hat{\gamma}'\|_R^2\quad\hbox{for some $C>0$.}
\end{equation}
Thus, taking into account (\ref{h}) and (\ref{hh}) in (\ref{hhh}), we deduce
\begin{equation}\label{g}
(n-1)(\log\beta\circ\hat{\gamma})''\leq
C\,\left(\frac{1}{\mu}+1\right)\hat{t}'^2+D.
\end{equation}
Next, we integrate both sides of (\ref{g}) along $[0,\hat{s}]$, for any $0\leq\hat{s}<\hat{l}$,
\begin{equation}\label{z}
(\log\beta\circ\hat{\gamma})'\mid_0^{\hat{s}}\leq C'\int_0^{\hat{s}}\hat{t}'^2 + D'.
\end{equation}
If the integral in the previous inequality remains bounded for all $\hat{s}\in [0,\hat{l})$, then we have a contradiction with point 3) above. Hence, we can assume that it is not bounded above, and so, since $\hat{\gamma}$ is future-directed, $\hat{t}'$ is (positive and) unbounded from above. On the other hand, note that
\begin{equation}\label{zz}
\int_{0}^{\hat{s}}\hat{t}'^2=(\hat{t}(\hat{s})\hat{t}'(\hat{s})-\hat{t}(0)\hat{t}'(0))-\int_0^{\hat{s}}\hat{t}\hat{t}''.
\end{equation}
% \sout{On the other hand, since $\hat{\gamma}$ is future-directed, the unboundedness of the integral in (\ref{z}) means in particular that $\hat{t}'$ is (positive and) unbounded from above.}
We need to show that 
\begin{equation}\label{zzz}
\int_{0}^{\hat{l}}\hat{t}\hat{t}''>0.
\end{equation}
As $\hat{t}$ is continuous, it is bounded on $M$, and it is strictly increasing, so there is no loss of generality in assuming it is positive on $[0,\hat{l})$. Let 
\[
J_{+}:= \{ s \in (0,\hat{l}) \, : \, \hat{t}''(s) >0\}
\]
and
\[
J_{-}:= \{ s \in (0,\hat{l}) \, : \, \hat{t}''(s) <0\}. 
\]
Proving (\ref{zzz}) is equivalent to showing that
\begin{equation}
\label{compare}
\int _{J_{+}}\hat{t}(s)|\hat{t}''(s)| ds > \int _{J_{-}}\hat{t}(s)|\hat{t}''(s)| ds . 
\end{equation}
But because $\hat{t}'$ diverges to infinity on $[0,\hat{l})$, it is not difficult to see that for any interval $I\subset [0,\hat{l})$ with $\hat{t}''<0$ on $I$ and sup$(I)\leq \hat{t}_0$ for some $\hat{t}_0\in [0,\hat{l})$, there exists some subset $J_I\subset [0,\hat{l})$ with $\hat{t}_0\leq$inf$(J_I)<$sup$(J_I)<\hat{l}$ such that
$\int_{J_I}\hat{t}''>-\int_I\hat{t}''$; for example, take $\hat{t}_0<\hat{s}_1<\hat{s}_2<\hat{l}$ such that $\hat{t}'(s_2)-\hat{t}'(s_1)>-\Delta\hat{t}'\mid_I$. Define then $J_I:=\{\hat{s}\in[\hat{s}_1,\hat{s}_2]:\hat{t}''(\hat{s})>0\}$. Since inf$(J_I)>$sup$(I)$,
$\int_{J_I}\hat{t}\hat{t}''>\int_{J_I}{\rm inf}(J_I)\hat{t}''>-\int_I{\rm sup}(I)\hat{t}''>-\int_I\hat{t}\hat{t}''$. Hence (\ref{compare}), and hence (\ref{zzz}), holds. 

Next, we replace (\ref{zz}) into (\ref{z}), and take into account (\ref{zzz}) and the boundedness of $\hat{t}$ on $[0,\hat{l})$, and deduce
\begin{equation}\label{n}
(\log\beta\circ\hat{\gamma})'(\hat{s})\leq C''\hat{t}'(\hat{s})
%-C'\int_{0}^{\hat{s}}\hat{t}\hat{t}''
+D''\quad\hbox{for $\hat{s}$ near $\hat{l}$.}
\end{equation}
Thus, if we integrate both sides of (\ref{n}), and use again that $\hat{t}(\hat{s})$ remains bounded on $[0,\hat{l})$, we arrive at a contradiction with point (3) above. \qcd

\section*{Acknowledgments}

The authors are partially supported by the Spanish Grant MTM2016-78807-C2-2-P (MINECO and FEDER funds). The second author also wishes to acknowledge the Department of Mathematics, Universidade Federal de Santa Catarina (Brazil), for the kind hospitality while part of this research was being carried out.

\bibliographystyle{unsrt}

\begin{thebibliography}{1}

\bibitem{B} R. Bartnik, {\it Remarks on cosmological spacetimes and constant mean curvature
surfaces}, Comm. Math. Phys. {\bf 117} (1988), no. 4, 615–624.
%
\bibitem{BE}
J.K. Beem, P.E. Ehrlich and K.L. Easley
  {\it Global Lorentzian Geometry}, $2^{\rm{nd}}$ ed.,
  Marcel Dekker, New York (1996).
%
%\bibitem{DE}
%P. Brax, {\em What makes the Universe accelerate? A review on what dark energy could be and how to test it}, Rep. Prog. Phys. {\bf 81}, (2018) 016902.
%
%
\bibitem{besse}
A.L. Besse, {\em Einstein manifolds}, Springer-Verlag, Berlin Heidelberg (1987).
%
\bibitem{CFS} A.M. Candela, J.L. Flores, M. S\'anchez, {\it Global hyperbolicity and Palais–Smale condition
for action functionals in stationary spacetimes}, Adv. Math. {\bf 218} (2008), 515–536.
%
\bibitem{splittingus}
I.P. Costa e Silva and J.L. Flores, {\em On the splitting problem for Lorentzian
manifolds with an $\mathbb{R}$-action with causal orbits}, Ann. Henri Poincar\'e {\bf 18} (2017), no. 5, 1635–1670.
%
\bibitem{E}
J.-H. Eschenburg, {\em The splitting theorem for space-times with strong energy con\-di\-tion}, J. Differential Geom. 27 (1988), no. 3, 477–491.
%
\bibitem{EG}
J.-H. Eschenburg and G. J. Galloway, {\em Lines in space-times}, Comm. Math. Phys.
{\bf 148} (1992), no. 1, 209–216.
%
\bibitem{G1}
G. J. Galloway, {\em Splitting theorems for spatially closed space-times}, Comm. Math.
Phys. {\bf 96} (1984), no. 4, (1984) 423–429.
%
\bibitem{G2}
G. J. Galloway, {\em Some rigidity results for spatially closed spacetimes}, Mathematics of
gravitation, Part I (Warsaw, 1996), Banach Center Publ., vol. 41, Polish Acad.
Sci., Warsaw, (1997) pp. 21–34.
%
\bibitem{GV1} G.J. Galloway, C. Vega, {\it Hausdorff closed limits and rigidity in Lorentzian geometry}, Ann. Henri
Poincar\'e {\bf 18} (2017), no. 10, 3399–3426.
%
\bibitem{GV2} G.J. Galloway, C. Vega, {\it Rigidity in vacuum under conformal symmetry}, Letters in Math. Phys. {\bf 108} (2018) 10, 2285-2292.
%
\bibitem{GKP}
R.P. Geroch, E. H. Kronheimer. and R. Penrose, {\em Ideal points in spacetime},
Proc. Royal Soc. Lond. {\bf A327}, (1972) 545-567.
%
\bibitem{GH}
D. Garfinkle and S.G. Harris, {\it Ricci fall-off in static and stationary, globally
hyperbolic, non-singular spacetimes}, Class. Quantum Grav. {\bf 14} (1997), no. 1,
139–151.
%
\bibitem{HL}
S.G. Harris and R.J. Low, {\em Causal monotonicity, omniscient foliations and the shape of space}, Class. Quantum Grav. {\bf 18} (2001) 27-43.
%
\bibitem{HE}
S.W. Hawking and G.F.R. Ellis,
  {\it The Large Scale Structure of Space-time},
  Cambridge University Press, Cambridge (1973).
%
\bibitem{HP}
S.W. Hawking and R. Penrose, {\em The singularities of gravitational
collapse and cosmology}, Proc. Royal Soc. Lond. {\bf A314}, (1970) 529-548.
%
\bibitem{JS}
M.A. Javaloyes and M. S\'{a}nchez, {\em A note on the existence of standard splittings for conformally stationary spacetimes}, Class. Quantum Grav. {\bf 25} (2008) 168001.
%
\bibitem{KN}
S. Kobayashi and K. Nomizu, {\em Foundations of Differential Geometry}, Vol. 1, Wiley, New York (1962).
%
\bibitem{Lee}
J.M. Lee, {\em Introduction to smooth manifolds}, Springer, New York (2003).
%
\bibitem{oneill}
B. O'Neill,
  {\it Semi-Riemannian Geometry with Applications to Relativity},
  Academic Press, New York (1983).
%
%\bibitem{perlmutter}
%S.J. Perlmutter {\em et al.}, {\em Measurements of $\Omega$ and $\Lambda$ from $42$ high-redshift supernovae}, Astrophys. J. {\bf 517}, (1999) 565-586.
%
%\bibitem{riess}
%A.G. Riess {\em et al.}, {\em Observational evidence from supernovae for an accelerating universe and a cosmological constant}, Astron. J., {\bf 116}, (1998) 1009-1038.
%
\bibitem{RS}
A. Romero and M. S\'anchez, {\em On completeness of certain families of semi-Riemannian manifolds}, Geom. Dedicata, {\bf 53} (1994), no. 1, 103-117.
%
%\bibitem{flat}
%D.N. Spergel {\em et al.}, {\em Three-year Wilkinson Microwave Anisotropy Probe (WMAP) observations: implications for cosmology}, Astrophys. J. Suppl., {\bf 170}, (2007) 377-408.
%
%\bibitem{DM}
%T.J. Sumner, {\em Experimental searches for dark matter}, Living Rev. Relativ. (2002) 5: 4. https://doi.org/10.12942/lrr-2002-4.
%
\bibitem{Vandenban}
E.P. van den Ban, {\em Lecture notes on Lie groups}, available at http://www.staff.science.uu.nl/~ban00101/lecnotes/lie2010.pdf.
%
\bibitem{wald} R.M. Wald, {\em General Relativity}, University of Chicago Press, Chicago (1984).
%
\bibitem{yau}
S.-T. Yau, {\em Problem section}, Annals of Math. Studies, No. 102, Princeton University
Press, Princeton, N. J., (1982) pp. 669–706.

\end{thebibliography}

\end{document}